\documentclass{amsart}
\usepackage[all]{xy}
\usepackage{amsmath}
\usepackage{amsfonts}
\usepackage{amssymb}
\usepackage{latexsym}
\usepackage{amscd}
\usepackage{latexsym}
\usepackage[dvips]{graphicx}
\usepackage{mdwtab}
\usepackage{float}

\usepackage{hyperref}
\usepackage{enumitem}
\usepackage[dvipsnames]{xcolor}
\newtheorem{theo}{Theorem}

\newtheorem{prop}[theo]{Proposition}
\newtheorem{cor}[theo]{Corollary}
\newtheorem{defi}{Definition}

\newtheorem{rem}{Remark}

\newcommand{\beq}{\begin{equation}}
\newcommand{\eeq}{\end{equation}}

\newcommand{\Z}{{\mathbb{Z}}}
\newcommand{\R}{{\mathbb{R}}}

\newcommand{\C}{{\mathbb{C}}}

\newcommand{\ind}{$\operatorname{ind}_{\Z_2}$}
\newcommand{\type}{\mbox{\LARGE{$\epsilon$}}}
\newfloat{figure}{thp}{Figure}
\floatname{figure}{\bf Figure}
\newfloat{table}{thp}{Table}[section]
\floatname{table}{\bf Table}

\usepackage{multirow}
\usepackage{array}

\begin{document}

\title[The Borsuk-Ulam theorem for 
 $3$-manifolds]{The Borsuk-Ulam theorem for 
closed $3$-manifolds having geometry $S^2\times \R$}


\author{A. Bauval, D. L.\ Gon\c calves, C. Hayat, and P. 
Zvengrowski} 


\begin{abstract}

\emph{Let $M$ be a closed  3-manifold  which admits  the   geometry  $S^2\times \R$. In this 
work we determine all the free involutions $\tau$ on $M$, and the  Borsuk-Ulam  index  of
$(M,\tau)$. }
 
\end{abstract}

\maketitle

\section{Introduction}

The theorem now known as the Borsuk-Ulam theorem seems to have first appeared in a paper by Lyusternik and
Schnirel'man \cite{ls} in 1930, then in a paper by Borsuk \cite{bo} in 1933 (where a footnote mentions that the theorem was posed as a conjecture by S. Ulam). One of the most familiar statements
(Borsuk's Satz II) is that for any continuous map $f\colon S^n\to\R^n,$ \ there exists a point
\ $x \in S^n$ \ such that \ $f(x) = f(-x)$. The theorem has many equivalent forms and generalizations,
an obvious one being to replace $S^n$ and its antipodal involution \  $\tau(x) = -x$\ by 
any finite-dimensional connected
CW-complex $X$ equipped with some fixed point free involution $\tau$, and ask whether \ $f(x) = f(\tau(x))$ \ 
must hold for some \ $x \in X$.\  The original theorem and its generalizations have many applications
in topology, and also -- since Lov\'asz's \cite{lov} and B\'ar\'any's \cite{aa} pioneering work in 1978 -- in combinatorics and graph theory. An excellent general reference is Matou\v sek's book \cite{mat}. For more about some new developments in the subject, as  well the terminology used here, see \cite{BauHaGoZv1}.
 
In recent years, the following generalization of the question raised by Ulam has been studied for many families of pairs $(M, \tau)$, where $\tau$ is a free involution on the space $M$:
 Given $(M, \tau)$, determine all positive integers $n$ such that
 for every map\ $f\colon M \to \R^n$,\  there  is an $x\in M$ for which $f(x)=f(\tau(x))$.
When $n$ belongs to this family, we say that   the pair $(M, \tau)$ has the Borsuk-Ulam property with respect to  maps into $\R^n$.  Not surprisingly, examples of spaces are known for which the value of $n$ depends not only on the space but also on the free involution (assuming such exists), 
cf. \cite{gon}, \cite{ghz},
 or \cite{BauHaGoZv1}, as well as the present note.
 
 The Borsuk-Ulam property has been studied for several families of closed (compact without boundary) manifolds,  the spheres 
  being the first such family.
 The problem of existence and classification of free involutions  arises naturally for a manifold $M$.  In particular various results have been obtained for 3-manifolds, cf. \cite{BauHaGoZv1}.

In this note we shall consider closed 3-manifolds admitting the geometry $S^2 \times \R$. Our problem, then, is to classify
 the free involutions $\tau$ (up to equivalence, see Definition \ref{defi:equiv} below) on each $M$ under consideration, and to determine, for such a pair $(M, \tau)$, the integers $n$ for which the pair has the Borsuk-Ulam property with respect to  maps into $\R^n$.
 
 \begin{rem}\label{rem:7covered}

 Since the seven 3-manifolds admitting the geometry $S^2\times\R$ are precisely the manifolds 
 covered by  $S^2\times\R$ 
\cite[p. 458]{scott},
 it follows that for any covering projection $M \to N$ of $3$-manifolds, if one of the two admits the geometry $S^2\times\R$ then so does the other. 
  \end{rem}
  
 \begin{rem}\label{rem:4closed}

Among these seven 3-manifolds, exactly four are closed \cite{tollefson1},
namely:
\begin{enumerate}
\item$S^2\times S^1$,
\item the non-orientable $S^2$-bundle $E$ over $S^1$,
\item $\mathbb{R}P^2\times S^1$, and
\item $\mathbb{R}P^3\#\mathbb{R}P^3,$ the connected sum of two $\mathbb{R}P^3$.
\end{enumerate}

These four manifolds are naturally Seifert-fibred by taking a suitable group of isometries of  $S^2\times\R$ \cite[p. 457]{scott}.
 \end{rem}
%


The goal of the present work is to answer the problem above for the four Seifert manifolds of Remark \ref{rem:4closed}, in the same spirit as we did for flat- and nil-geometries (\cite{BauHaGo1}, \cite{BauHaGo2}). For most of the remaining Seifert manifolds\footnote{The only presently unknown case is that of manifolds which admit geometry $\mathcal{H}^2\times\R$, with an involution such that the quotient is not a Seifert
manifold in the sense of Seifert's definition.} the same problem can be solved similarly, using \cite{BauHa1} and \cite{BauHa2}.

  
 
       For the closed manifolds  which admit Sol geometry, some partial
results are known, see \cite{BarGonVen}. 
For hyperbolic 3-manifolds not much is presently known about the Borsuk-Ulam property.
 
 \bigskip
 
In order to determine the pairs $(M, \tau)$, where $M$ runs over the double coverings of a given manifold $N$,
we compute $\pi_1(M)=\ker\varphi$, for all possible epimorphisms $\varphi:\pi_1(N)\twoheadrightarrow\Z_2$.
The involution $\tau$ is then the involution associated to the double cover, and $M/\tau$ is homeomorphic to $N$. The characteristic class
$[\varphi]\in H^1(N;\Z_2)$ of the fibre bundle $M\twoheadrightarrow N$ determines the answer to the Borsuk-Ulam problem for $(M,\tau)$.
 For the four closed 3-manifolds supporting $S^2 \times \R$ geometry, the main result is Theorem 4, which is also summarized in Figures 1,2.
The cases A) and B) of Theorem \ref{thm:4}  were solved recently,
using a slightly different approach, in \cite[ Theorem 17 and Proposition 18]{BlMa}.
  
 
 This work contains two sections besides the present Introduction. 
In Section \ref{sec:Preliminaries}, we give more details about the  Borsuk-Ulam property and the methods  used 
   to study it, as well as some details about Seifert manifolds.
In Section \ref{sec:The Borsuk-Ulam theorem for the geometry}, we solve the problem for the closed manifolds  having geometry $S^2\times \R$.


\section{Preliminaries}\label{sec:Preliminaries}
Let us recall some known results that will be used throughout this paper. We henceforth assume that $M$ is a connected 3-manifold, $\tau$ a free involution on it, $N$ the quotient manifold ($N=M/\tau$), and $\varphi : \pi_1(N) \twoheadrightarrow \Z_2$ the associated epimorphism classifying the principal $\Z_2$-bundle  $M\twoheadrightarrow N$.

\begin{defi}\label{defi:bu}~
\begin{itemize}
\item The pair $(M,\tau)$ satisfies the  Borsuk-Ulam property for $\R^n$ if for any continuous map $f:M\to\R^n$, there is at least one point $x\in M$ such
that $f(x) = f(\tau(x))$.
\item The $\Z_2$-index of $(M,\tau)$, denoted by \ind$(M, \tau)$, is the greatest integer $n$ such that $(M,\tau)$  satisfies the  Borsuk-Ulam property for $\R^n$.
\end{itemize}
\end{defi}

From \cite{yang} and \cite{ghz}, it is known that $1\le$\ind$(M,\tau)\le{\rm dim}(M)$, and from Theorems 3.1 and 3.2  proved in \cite{ghz},  we have:
 
\begin{theo}\label{thm:p2}~
\begin{enumerate}[label=(\roman*)]
\item\ind$(M,\tau)=1$ if and only if the epimorphism $\varphi:\pi_1(N)\twoheadrightarrow\Z_2$ factors through the projection $\Z\twoheadrightarrow\Z_2 $.
\item\ind$(M,\tau)=3$ if and only if the cup-cube $[\varphi]^3$ is non-zero, where
$[\varphi] \in H^1(N; \Z_2)$ is the characteristic class  of the double covering (principal $\Z_2$-bundle) $M\twoheadrightarrow N$ associated to $\varphi$. 
\end{enumerate}
\end{theo} 

Next, let us define an equivalence relation on the set of pairs $(M, \tau)$.

\begin{defi}\label{defi:equiv}


For $i=1$ or $2$, let $M_i$ be a connected 3-manifold, $\tau_i$ a free involution on it, $N_i$ the quotient manifold ($N_i=M_i/\tau\_i$), and $\varphi_i$ the associated epimorphism ($\varphi_i:\pi_1(N_i)  \twoheadrightarrow\Z_2$).

$(M_1,\tau_1)$ and $(M_2,\tau_2)$ are equivalent -- or $\varphi_1$ and $\varphi_2$ are equivalent -- if the two bundles $M_i\to N_i$ are isomorphic.

\end{defi}

In the sequel, most pairs will be trivially non-equivalent because the base spaces of their associated bundles are non-homeomorphic.



Henceforth, $M$ is one of the four manifolds $S^2\times S^1$, $E$, $\mathbb{R}P^2\times S^1$, or $\mathbb{R}P^3\#\mathbb{R}P^3$
of Remark \ref{rem:4closed}. Following the notation of Orlik \cite{orlik}, $M$ will be 
described by a list of Seifert invariants written in the normal form $\{b;(\type, g)\}$,
hence without exceptional fibres (the type $\type$\ reflects the orientations of the base surface and the total space of the Seifert fibration of $M$, and the integer $g$ is the genus of the base surface -- the orbit space obtained by identifying each $S^1$  fibre of $M$ to a point). These invariants $\{b;(\type, g)\}$ provide the Seifert  presentations of the fundamental groups of our four particular manifolds:

\begin{itemize}
\item $\pi_1(S^2\times S^1)=\pi_1(0; (o_1,0))=\langle h\rangle\approx\Z$,\\
\item $\pi_1(E)=\pi_1(1; (n_1,1))=\langle v,h\mid v^2h^{-1}\rangle=\langle v\rangle \approx\Z$,\\
\item $\pi_1(\mathbb{R}P^2\times S^1)=\pi_1(0; (n_1,1))=\langle v,h\mid v^2,vhv^{-1}h^{-1}\rangle \approx \Z_2\times\Z$,\\
\item $\pi_1(\mathbb{R}P^3\#\mathbb{R}P^3)=\pi_1(0; (n_2,1))=
\langle v,h\mid v^2,(vh)^2\rangle\approx\Z_2*\Z_2$.
\end{itemize}

We extract from \cite[Proposition 4.2]{BauHaGoZv1} the restricted cases needed here:

%
%

\begin{prop}\label{prop:cube}    Let $N$ be one of the four manifolds $S^2\times S^1, E, \R P^2 \times S^1,$ or $ \R P^3 \# \R P^3$, and $\varphi, [\varphi]$
as in Theorem 1 above. Then, $[\varphi]^3=0$ except in the two following cases:
\begin{itemize}
\item $N=\mathbb{R}P^2\times S^1$ and $\varphi(h)+\varphi(v)\ne  0$;
\item $N=\mathbb{R}P^3\#\mathbb{R}P^3$ and $\varphi(h)\ne  0$.
\end{itemize}
\end{prop}

Using this proposition and our presentation of $\pi_1(N)$, Theorem \ref{thm:p2} gives:
%
%
%
%
%

{\begin{cor}\label{cor:Ind44}  With $M, N, \tau, \varphi$ as above, \ind$(M, \tau)$  equals:
\begin{itemize}
\item$1$ if $N=S^2\times S^1$ or $E$, or if $N=\mathbb{R}P^2\times S^1$ and $\varphi(v)=\  0$;
\item$2$ if $N=\mathbb{R}P^2\times S^1$ or $\mathbb{R}P^3\#\mathbb{R}P^3$ and $\varphi(h)= 0$;
\item$3$ if either $N=\mathbb{R}P^2\times S^1$ and $\varphi(h)+\varphi(v)\ne 0$, or
 $N=\mathbb{R}P^3\# \mathbb{R}P^3$ and $\varphi(h)\ne  0$.
\end{itemize}
\end{cor}}
 

\section{The Borsuk-Ulam theorem for the geometry ${\bf S^2\times \R}$}\label{sec:The Borsuk-Ulam theorem for the geometry}
\label{exceptioS}

Remark 1 and Remark 2 imply that the four manifolds under consideration are exactly the 3-manifolds covered by $S^2\times S^1$ (this is also illustrated in Figure 1), and also that this family is closed  under taking finite covers or finite quotients.

\bigskip 

We now state the main result (recall that             ``unique'' here means up to the equivalence of Definition 2).

%
%
%

\begin{theo} \label{thm:4}Up to equivalence,

{\bf A)} The space $S^2\times S^1$ admits four free involutions $\tau_1, \tau_2, \tau_3, \tau_4$ such that:

\begin{enumerate}

\item The quotient $(S^2\times S^1)/\tau_1$ is homeomorphic to $S^2\times S^1$, and \ind$(S^2\times S^1,\tau_1)=1$.  

\item The quotient $(S^2\times S^1)/\tau_2$ is homeomorphic to $E$, and \ind$(S^2\times S^1,\tau_2)=1$.

\item The quotient $(S^2\times S^1)/\tau_3$ is homeomorphic to $\mathbb{R}P^2\times S^1$, and \ind$(S^2\times S^1,\tau_3)=2$.

\item The quotient $(S^2\times S^1)/\tau_4$ is homeomorphic to $\mathbb{R}P^3\#\mathbb{R}P^3$, and \ind$(S^2\times S^1,\tau_4)=2$.
   
\end{enumerate}

 {\bf B)} The space $E$ admits a unique free involution  $\tau_5$.
 The quotient $E/\tau_5$ is homeomorphic to $\mathbb{R}P^2\times S^1$, and \ind$(E,\tau_5)=3$.

 {\bf C)} The space $\mathbb{R}P^2\times S^1$ admits a unique free involution 
$\tau_6$.
 The quotient $(\mathbb{R}P^2\times S^1)/\tau_6$ is homeomorphic to $\mathbb{R}P^2\times S^1$, and \ind$(\mathbb{R}P^2\times S^1,\tau_6)=1$.

{\bf D)} The space $\mathbb{R}P^3\#\mathbb{R}P^3$ admits a unique free involution $\tau_7$.
The quotient $(\mathbb{R}P^3\#\mathbb{R}P^3)/\tau_7$ is homeomorphic to  $\R P^3\#\ R P^3$, and \ind$(\mathbb{R}P^3\#\mathbb{R}P^3,\tau_7)=3$.

\end{theo}

 The rest of this section is devoted to the  proof of this theorem, and the involutions $\tau_1,\ldots,\tau_7$ will be specified in Propositions \ref{lem:S2S1}-\ref{lem:RP}.

\subsection{2-coverings and $\Z_2$-index}

The purpose of this subsection is to present the  
graph below where the arrows are directed from the covering to the base. The number associated with each arrow is the $\Z_2$-index of the covering. This graph is a summary of Theorem \ref{thm:4}. For example, the arrow from $S^2\times S^1$ to $E$ represents Theorem \ref{thm:4} (A2), and the $1$ along the arrow corresponds to \ind$(S^2\times S^1, \tau_2)=1$.


~
\vspace{1.5cm}

\begin{figure}[!h]
\setlength{\unitlength}{10pt}
\begin{center}
\begin{picture} (10,10)(-8,-8) 
\thicklines

\put(-5,0){
\qbezier[300](-2,4.7)(-1,3)(2,5)
\qbezier[300](-2,5.3)(-1,7)(2,5)}
\put(-7,4.7){\vector(0,1){.7}}

\put(-2.9,5){\line(0,-3){3}}
\put(-2.9,-.7){\line(0,1){2.5}}
\put(-2.9,1.3){\vector(0,1){.7}}

\put(-5,0){
\qbezier[300](2.1,-.7)(3,-3)(2,-4.85)
\qbezier[300](2.1,-.7)(1,-3)(1.35,-4.85)
\qbezier[300](2,-4.85)(1.6,-6)(1.35,-4.85)
\put(1.4,-3){\vector(0,1){0.75}}
\put(2.2,-.7){\line(1,0){4}}
\put(5.2,-.7){\vector(1,0){.9}}
\put(6,-.7){\line(1,0){4}}
\put(2.2,-.7){\line(4,3){3.7}}
\put(10,-.7){\line(-3,2){4.1}}
\put(3.2,-0.1){\vector(1,1){.9}}
\put(7,1.46){\vector(1,-1){.9}}
\qbezier[300](10.1,-.7)(11.6,-2.8)(11,-4.85)
\qbezier[300](10.1,-.7)(9.5,-3)(10.35,-4.85)
\qbezier[300](11,-4.85)(10.75,-6)(10.3,-4.85)
\put(9.8,-3){\vector(0,1){0.75}}
}

\put(-3.3,5.5){\small $\mathbb{R}P^3\#\mathbb{R}P^3$}
\put(0.6,2.3){\small $E$}
\put(-6.3,-1){\small $S^2\times S^1$}
\put(5.3,-1){\small $\mathbb{R}P^2\times S^1$}

\put(-6.5,4.8){\tiny $3$}
\put(-3.8,1.3){\tiny $2$}
\put(-3.2,-2.9){\tiny $1$}
\put(5.2,-2.9){\tiny $1$}
\put(0.6,-.3){\tiny $2$}
\put(-0.9,0){\tiny $1$}
\put(2.41,0){\tiny $3$}

\end{picture}
   \caption{\label{fig:les4} Graph for the geometry ${\bf S^2\times \R}$.}
   \end{center}
  \end{figure}
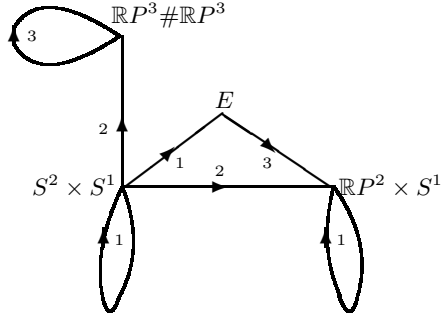

As before, let $\varphi:\pi_1(N)\to \Z_2$ be an epimorphism. In order to obtain the 2-covering $M$ of $N$ determined by $\varphi$,
we compute $\pi_1(M)=\ker\varphi$.
Each case of Theorem \ref{thm:4} is proved by one of the following propositions; this correspondence is given in Figure 2 below. In all cases the presentation of $\pi_1(N)$ is as in Section \ref{sec:Preliminaries}. We shall generally denote points of $S^1$ as $z\in \C, |z|=1$, points of $S^2$ as $x\in \R^3, \| x \| =1$, and points of $\R P^2$ as $[x]= \{ \pm x \}, x \in S^2$. 

\begin{prop}\label{lem:S2S1}
\begin{enumerate} 
\item   There is
  a unique epimorphism from $\pi_1(S^2\times S^1)$ onto $\Z_2$.
\item The associated $2$-covering and free involution form the pair $(S^2\times S^1,\tau_1)$, where $\tau_1(x,z) = (x,-z)$.
\item Its $\Z_2$-index is $1$.
\end{enumerate}
\end{prop}
{\it Proof.} \begin{enumerate} 
\item There is a unique epimorphism from $\Z$ onto $\Z_2$.
\item Clearly, $\tau_1$ is a free involution on $S^2\times S^1$, and $(S^2\times S^1)/\tau_1$ is homeomorphic to $S^2\times S^1$.
\item  This follows  from Corollary \ref{cor:Ind44}.         \ \ \ \ \ \ \ \ \ \ \ \ \ \ \ \ \ \ \ \       \ \ \ \ \ \ \ \ \ \ \ \ \        \ \ \ \ \ \ \ \ \ \ \ \ \ \ \ \ \ \ \ \ \ \ \ \ \ \ \ \ \ \ \ \ $\square$
\end{enumerate}  



\begin{prop}\label{lem:E}
\begin{enumerate} 
\item There
is a unique epimorphism from $\pi_1(E)$ onto $\Z_2$.
\item The associated $2$-covering and free involution form the pair $(S^2\times S^1,\tau_2)$, where $\tau_2(x,z)=(-x,-z)$.
\item Its $\Z_2$-index is $1$.
\end{enumerate}
\end{prop}

\begin{proof} Same arguments as in the previous proposition.
\end{proof}

\begin{prop}\label{lem:SRP} 
\begin{enumerate}
\item 
There are three epimorphisms 
from $\pi_1(\mathbb{R}P^2\times S^1)$ onto $\Z_2$
 defined by:
$$\begin{array}{ll}
 \varphi_1(v)= 0,&\varphi_1(h)=1\\  
  \varphi_2(v)=1,&  \varphi_2(h)=0\\ 
 \varphi_3(v)=1,& \varphi_3(h)=1. 
 \end{array}$$

\item 
\begin{enumerate}[label=(\roman*)]
\item The $2$-covering and free involution associated to $\varphi_1$ form the pair $(\mathbb{R}P^2\times S^1,\tau_6)$, where $\tau_6([x],z)=([x],-z)$.
Its $\Z_2$-index is $1$.
\item The $2$-covering and free involution associated to $\varphi_2$ form the pair $(S^2\times S^1,\tau_3)$, 
where $\tau_3(x,z)=(-x,z)$.
Its $\Z_2$-index is $2$.
\item The $2$-covering and free involution associated to $\varphi_3$ form the pair $(E,\tau_5)$, where $\tau_5([x,y])=[-x,y]$
(using the representation of $E$ as $S^2\times S^1/\tau_2$). Its $\Z_2$-index is $3$.
\end{enumerate}
\end{enumerate}
\end{prop}

\begin{proof}
\begin{enumerate} 
\item  This is clear.
\item 
\begin{enumerate}[label=(\roman*)]
\item One has $\ker\varphi_1=\langle v,h^2\rangle\approx\Z_2\times\Z$ and $(\mathbb{R}P^2\times S^1)/\tau_6$ 
 is homeomorphic to  $\mathbb{R}P^2\times S^1$.
By Corollary \ref{cor:Ind44}, \ind$(\mathbb{R}P^2\times S^1,\tau_6)=1$.
\item One has  $\ker\varphi_2
\approx\Z$, generated  by the element $(\overline0,1)\in \Z_2\times \Z$.
This corresponds to the product by $S^1$ of the orientation cover of $\mathbb{R}P^2$. The correctness of $(M,\tau_3)$ follows.
By Corollary \ref{cor:Ind44}, \ind$(S^2\times S^1,\tau_3)=2$.
\item It is straightforward to see that $E/\tau_5$ is homeomorphic to 
$\mathbb{R}P^2\times S^1$.
The only remaining possibility for the  classifying map of the bundle $E\to \mathbb{R}P^2\times S^1$ is $\varphi_3$, whence
by Corollary \ref{cor:Ind44}, \ind$(E,\tau_5)=3$.
\end{enumerate}
\end{enumerate}
\end{proof}

\bigskip

 
 



\bigskip

 \begin{prop}\label{lem:RP}
\begin{enumerate} 
\item 
There are three epimorphisms 
from $\pi_1(\mathbb{R}P^3\# \mathbb{R}P^3)$ onto $\Z_2$ defined by:
$$\begin{array}{ll}
  \varphi_1(v)=1,&\varphi_1(h)= 0\\  
  \varphi_2(v)=0,&  \varphi_2(h)=1\\ 
 \varphi_3(v)=1,& \varphi_3(h)=1.  
 \end{array}$$ 

\end{enumerate}


\bigskip

\item\begin{enumerate}[label=(\roman*)]
\item The $2$-covering and free involution associated to $ \varphi_1$ form the pair $(S^2\times S^1,\tau_4)$, where 
$\tau_4(x,z)=(-x,\overline z)$.
Its $\Z_2$-index is $2$.
\item The epimorphisms $\varphi_2$ and $\varphi_3$ are equivalent (cf. Definition 2).
For both of them, the associated $2$-covering and free involution form the same pair $(\mathbb{R}P^3\#\mathbb{R}P^3,\tau_7)$, where $\tau_7([x,z])=[x,-z]$
(using the representation of $\mathbb{R}P^3\#\mathbb{R}P^3$ as $S^2\times S^1/\tau_4$). Its $\Z_2$-index is $3$.
\end{enumerate}

\end{prop}

\begin{proof}

(1)  This is clear.

(2) First 
 notice that by 
Corollary \ref{cor:Ind44}, 
\ind$(S^2\times S^1, \tau_4)=2$ and  \ind$(\mathbb{R}P^3 \# \mathbb{R}P^3,\tau_7)=~3$.
Also note that since $N$ is orientable, so are its double covers.

\begin{enumerate}[label=(\roman*)]
\item One has $\ker\varphi_1=\langle h\rangle\approx\Z$, hence $M=S^2\times S^1$.

The correctness of $\tau_4$ stems from the fact that $(S^2\times S^1)/\tau_4$  
is homeomorphic to $\mathbb{R}P^3 \# \mathbb{R}P^3$, which is a 
simple consequence of \cite[p. 457]{scott}.

\item One has $\ker\varphi_2=\langle v,h^2\rangle$ and $\ker\varphi_3=\langle hv,h^2\rangle$. These two subgroups of $\pi_1(N)$ are isomorphic to the group itself, hence the total space $M$ of both associated bundles over $N$ is $N$.

One easily checks that $\tau_7$ is a fixed point free involution and $M/\tau_7$ is homeomorphic to $N$. 
Representing $N=\R P^3\#\R P^3$ as $S^2\times S^1/\tau_4$, the 
homeomorphism $\sigma$ of this quotient given by $\sigma([x,z])=[x,\overline 
z]$ exchanges the two copies of $\R P^3$,  hence it gives rise to two 
(obviously isomorphic) fibre bundles over $N$ : if $p:M\to N$ is one of 
them, the other one is $\sigma\circ p$. They correspond to $\varphi_2$ and 
$\varphi_3$ (in this order or in reverse order, depending on which 
homeomorphism between $M/\tau_7$  and $N$ was chosen to define $p$).
 \end{enumerate}
 
 \end{proof}
 
 In addition to Figure 1 above, the following Figure 2 also gives a summary of Propositions \ref{lem:S2S1} to  \ref{lem:RP}, from a slightly different perspective.

\vskip 1.0cm

$$\begin{array}{|c|c|c|c|c|c|}
 \hline\hline $Theorem \ref{thm:4}$&
$Proposition$& M&N&$Involution$&$Index$\\
  \hline
$(A1)$&$Proposition \ref{lem:S2S1}$& S^2\times S^1&S^2\times S^1&\tau_1&1\\
\hline
$(A2)$&$Proposition \ref{lem:E}$&   S^2\times S^1  &E  &\tau_2&1\\
\hline
  $(A3)$ &$Proposition \ref{lem:SRP} (2)(ii)$&     S^2\times S^1   & \ \ \ \mathbb{R}P^2\times S^1           &  \tau_3 &2 \\
    \hline
  $(A4)$&$Proposition \ref{lem:RP} (2)(i)$& S^2\times S^1 & \mathbb{R}P^3\# \mathbb{R}P^3 &  \tau_4&2\\
\hline
         $(B)$&$Proposition \ref{lem:SRP} (2)(iii)$& E &\mathbb{R}P^2\times S^1                       &    \tau_5 &3\\
        \hline
$(C)$&$Proposition \ref{lem:SRP} (2)(i)$&\mathbb{R}P^2\times S^1&\mathbb{R}P^2\times S^1 &\tau_6&1\\
\hline
                  $(D)$&$Proposition \ref{lem:RP} (2)(ii)$& \mathbb{R}P^3 \# \mathbb{R}P^3& \mathbb{R}P^3 \# \mathbb{R}P^3        &\tau_7&3\\
          \hline\hline
\end{array} $$

\smallskip

 \begin{center}
{\bf{Figure 2}}: Summary of Propositions \ref{lem:S2S1} to \ref{lem:RP}.
\end{center} 

\vskip 1.0cm

{\bf  Acknowledgment:} The second  author is partially supported  by the Projeto Tem\'atico-FAPESP  
``Topologia Alg\'ebrica, Geom\'etrica e Diferencial 2016/24707-4'' (Brazil).\\


\author{Anne Bauval}\\
\address{ \small Institut de Math\'ematiques de Toulouse (UMR 5219)\\
Universit\'e Toulouse III\\
118 Route de Narbonne, 31400 Toulouse - France\\
e-mail: bauval@math.univ-toulouse.fr}

\author{Daciberg L.\ Gon\c calves}\\
\address{\small Departamento de Matem\'atica - IME-Universidade de S\~ao Paulo\\
Rua de Mat\~ao, 1010\\
CEP: 05508-090 - S\~ao Paulo - SP - Brasil\\
e-mail: dlgoncal@ime.usp.br}

\author{Claude Hayat}\\
\address{\small Institut de Math\'ematiques de Toulouse (UMR 5219)\\
Universit\'e Toulouse III\\
118 Route de Narbonne, 31400 Toulouse - France\\
e-mail: hayat@math.univ-toulouse.fr}

\author{Peter Zvengrowski}\\
\address{\small Department of Mathematics and Statistics\\
University of Calgary\\
Calgary, Alberta, Canada T2N 1N4\\
e-mail: zvengrow@gmail.com}

\end{document}